\documentclass[leqno]{amsart}
\usepackage{caption}
\usepackage{amsmath,amssymb,amsthm,mathtools,mathrsfs,tikz}
\usepackage{environ}
\usepackage{stmaryrd}
\usetikzlibrary{arrows}
\usetikzlibrary{positioning}
\usetikzlibrary{decorations.text}
\usetikzlibrary{decorations.pathmorphing}
\usetikzlibrary{decorations.pathreplacing}
\makeatletter
\newsavebox{\@brx}
\newcommand{\llangle}[1][]{\savebox{\@brx}{\(\m@th{#1\langle}\)}%
	\mathopen{\copy\@brx\kern-0.5\wd\@brx\usebox{\@brx}}}
\newcommand{\rrangle}[1][]{\savebox{\@brx}{\(\m@th{#1\rangle}\)}%
	\mathclose{\copy\@brx\kern-0.5\wd\@brx\usebox{\@brx}}}
\makeatother
\makeatletter
\newsavebox{\measure@tikzpicture}
\NewEnviron{scaletikzpicturetowidth}[1]{%
	\mathrm{Def}\tikz@width{#1}%
	\mathrm{Def}\tikzscale{1}\begin{lrbox}{\measure@tikzpicture}%
		\BODY
	\end{lrbox}%
	\pgfmathparse{#1/\wd\measure@tikzpicture}%
	\edef\tikzscale{\pgfmathresult}%
	\BODY
}
\makeatother
\usepackage{bbm}
\DeclarePairedDelimiter\norm{\lvert}{\rvert}

\makeatletter
\let\save@mathaccent\mathaccent
\newcommand*\if@single[3]{%
	\setbox0\hbox{${\mathaccent"0362{#1}}^H$}%
	\setbox2\hbox{${\mathaccent"0362{\kern0pt#1}}^H$}%
	\ifdim\ht0=\ht2 #3\else #2\fi
}
\newcommand*\rel@kern[1]{\kern#1\dimexpr\macc@kerna}
\newcommand*\widebar[1]{\@ifnextchar^{{\wide@bar{#1}{0}}}{\wide@bar{#1}{1}}}
\newcommand*\wide@bar[2]{\if@single{#1}{\wide@bar@{#1}{#2}{1}}{\wide@bar@{#1}{#2}{2}}}
\newcommand*\wide@bar@[3]{%
	\begingroup
	\def\mathaccent##1##2{%
		\let\mathaccent\save@mathaccent
		\if#32 \let\macc@nucleus\first@char \fi
		\setbox\z@\hbox{$\macc@style{\macc@nucleus}_{}$}%
		\setbox\tw@\hbox{$\macc@style{\macc@nucleus}{}_{}$}%
		\dimen@\wd\tw@
		\advance\dimen@-\wd\z@
		\divide\dimen@ 3
		\@tempdima\wd\tw@
		\advance\@tempdima-\scriptspace
		\divide\@tempdima 10
		\advance\dimen@-\@tempdima
		\ifdim\dimen@>\z@ \dimen@0pt\fi
		\rel@kern{0.6}\kern-\dimen@
		\if#31
		\overline{\rel@kern{-0.6}\kern\dimen@\macc@nucleus\rel@kern{0.4}\kern\dimen@}%
		\advance\dimen@0.4\dimexpr\macc@kerna
		\let\final@kern#2%
		\ifdim\dimen@<\z@ \let\final@kern1\fi
		\if\final@kern1 \kern-\dimen@\fi
		\else
		\overline{\rel@kern{-0.6}\kern\dimen@#1}%
		\fi
	}%
	\macc@depth\@ne
	\let\math@bgroup\@empty \let\math@egroup\macc@set@skewchar
	\mathsurround\z@ \frozen@everymath{\mathgroup\macc@group\relax}%
	\macc@set@skewchar\relax
	\let\mathaccentV\macc@nested@a
	\if#31
	\macc@nested@a\relax111{#1}%
	\else
	\def\gobble@till@marker##1\endmarker{}%
	\futurelet\first@char\gobble@till@marker#1\endmarker
	\ifcat\noexpand\first@char A\else
	\def\first@char{}%
	\fi
	\macc@nested@a\relax111{\first@char}%
	\fi
	\endgroup
}
\makeatother

\usepackage{bm}
\usepackage[colorlinks=true,linkcolor=blue,citecolor=blue]{hyperref}
\usepackage{enumitem}
\usepackage{csquotes}

\def\irr#1{{\rm  Irr}(#1)}

\newcounter{intro}

\newtheorem{thm}{Theorem}
\newtheorem{lem}[thm]{Lemma}

\theoremstyle{remark}
\newtheorem{rem}[thm]{Remark}
\theoremstyle{definition}

\title[GVZ-groups]{A characterization of GVZ groups in terms of fully ramified characters}

\author{Shawn T. Burkett}
\address{Department of Mathematical Sciences, Kent State University, Kent,
	Ohio 44242, U.S.A.} \email{sburkett@math.kent.edu}

\author{Mark L. Lewis}
\address{Department of Mathematical Sciences, Kent State University, Kent,
	Ohio 44242, U.S.A.} \email{lewis@math.kent.edu}

\date{\today}

\begin{document}

\begin{abstract}
In this paper, we obtain a characterization of GVZ-groups in terms of commutators and monolithic quotients.  This characterization is based on counting formulas due to Gallagher.
\end{abstract}

\maketitle
Throughout this paper, all groups are finite.  For a group $G$, we write $\irr G$ for the set of irreducible characters of $G$.  In this paper, we present a new characterization of GVZ-groups.  A group $G$ is a {\it GVZ-group} if every irreducible character $\chi \in \irr G$ satisfies that $\chi$ vanishes on $G \setminus Z(\chi)$.  

The term GVZ-group was introduced by Nenciu in \cite{AN12gvz}.  Nenciu continued the study of GVZ-groups in \cite{AN16gvz} and the second author further continued these studies in \cite{ML19gvz}.  In our paper \cite{ourpre}, we showed that GVZ-groups can be characterized in terms of another class of groups that have appeared in the literature.  

An element $g \in G$ is called {\it flat} if the conjugacy class of $G$ is $g[g,G]$.  In \cite{flatness}, they defined a group $G$ to be {\it flat} if every element in $G$ is flat.  In fact, groups satisfying this condition had been studied even earlier.  Predating each of these references, Murai \cite{OnoType} referred to such groups as {\it groups of Ono type}.  In \cite{flatness}, they proved that if $G$ is nilpotent and flat, then $G$ is a GVZ-group.  Improving this result, we prove in \cite{ourpre} that a group $G$ is a GVZ-group if and only if it is flat.

In this paper, we characterize GVZ-groups using fully ramified characters.  For a normal subgroup $N$ of $G$, we say that the character $\chi \in \irr G$ is {\it fully ramified} over $N$ if $\chi_N$ is homogeneous  and $\chi(g) = 0$ for every element $g\in G\setminus N$.  

Following the literature, a group $G$ is called {\it central type} if there is an irreducible character of $G$ that is fully ramified over the center $Z(G)$.  Results about central type groups are in \cite{DJ}, \cite{espuelas}, \cite{gagola}, and \cite{IH}. 

With this as motivation, we define an irreducible character $\chi$ of $G$ to be {\it central type} if $\chi$, considered as a character of $G/\ker(\chi)$, is fully ramified over $Z(G/\ker(\chi))$.  (I.e., $G/\ker{\chi}$ is a group of central type with faithful character $\chi$.)  It is not difficult to see that $G$ is a GVZ-group if and only if every character $\chi \in \irr G$ is of central type.

Recall from the literature that a group is called {\it monolithic} if it has a unique minimal normal subgroup. It is easy to see that if $N$ is a normal subgroup of $G$ and $G/N$ is monolithic, then $N$ appears as the kernel of some irreducible character of $G$. Also an irreducible character $\chi$ is called {\it monolithic} if the quotient group $G/\ker(\chi)$ is monolithic. Thus, monolithic quotients correspond to monolithic characters.

The purpose of this paper is to give a new characterization of central type characters based on ideas of Gallagher that are encapsulated in \cite[Theorem 1.19 and Lemma 1.20]{MI18}, thereby obtaining a new characterizations of GVZ-groups. In particular, we prove the following theorem.

\begin{thm}\label{gvz thm}
Let $G$ be a nonabelian group. Then the following are equivalent: 
\begin{enumerate}[label={\rm(\arabic*)}]
\item $G$ is a GVZ-group.  
\item For every monolithic character $\chi \in \irr G$ and for every element $g \in G \setminus Z (\chi)$, there exists an element $x \in G$ so that $[g,x] \in Z(\chi) \setminus \ker (\chi)$.
\item $G$ is nilpotent, and for every normal subgroup $N$ of $G$ for which $G/N$ is monolithic and for every element $g \in G$ satisfying $[g,G]\nleq N$, there exists an element $x\in G$ such that $[g,x]\notin N$ and $[[g,x],G]\le N$.
\end{enumerate}
\end{thm}

Our proof relies on the following lemma, which we will see is an immediate consequence of some arguments of Gallagher that can be found in \cite[Theorem 1.19 and Lemma 1.20]{MI18}.  For an element $g \in G$, we set $D_G (g) = \{ x \in G \mid [x,g] \in Z(G) \}$.  Observe that $D_G (g)/Z(G) = C_{G/Z(G)} (g Z(G))$, so $D_G (g)$ is always a subgroup of $G$. 

\begin{lem}\label{fr 1}
Let $G$ be a group.  If the character $\vartheta \in \irr {{Z} (G)}$ is faithful, then $\vartheta$ is fully ramified with respect to $G/{Z}(G)$ if and only if $[g,D_G(g)] \ne 1$ for every element $g \in G \setminus {Z} (G)$.
\end{lem}

\begin{proof}
By Theorem 1.19 and Lemma 1.20 of \cite{MI18}, the number of irreducible constituents of $\vartheta^G$ equals the number of conjugacy classes of cosets $g Z(G) \in G/Z(G)$ that satisfy $[g,D_G(g)] = 1$.  Observe that if $g \in Z(G)$, then $[g,D_G(g)] = 1$.  Hence, the only way that there can be only one conjugacy class of elements of in $G/Z(G)$ satisfying this condition is if $[g,D_G(g)] \ne 1$ for all elements $g \in G \setminus Z(G)$. Since $\vartheta$ is fully ramified with respect to $G/Z(G)$ if and only if $\vartheta^G$ has a unique irreducible constituent, it follows that $\vartheta$ is fully ramified with respect to $G/Z(G)$ if and only if there is only one conjugacy class satisfying the condition.  This gives the desired result. 
\end{proof}

We get a slightly stronger statement without much difficulty.

\begin{lem}\label{fullyram}
Let $G$ be a group. If $\lambda \in \irr {{Z} (G)}$ is a character, then $\lambda$ is fully ramified with respect to $G/{Z} (G)$ if and only if $[g,D_G(g)] \nleq \ker(\lambda)$ for every element $g \in G \setminus Z(G)$.
\end{lem}

\begin{proof}
Let $Z = {Z} (G)$ and let $K = \ker (\lambda)$.  Suppose first that $\lambda$ is fully ramified with respect to $G/Z$.  Since $\lambda$ is fully ramified with respect to $G/Z$, it follows that $Z/K = {Z} (G/K)$.  Applying Lemma \ref{fr 1}, we have that $[gK,D_{G/K} (gK)] \ne 1$ for all cosets $gK \in G/K \setminus Z/K$.  It is not difficult to see that this implies that $[g,D_G (g)] \nleq K$ for all elements $g \in G \setminus Z$.  Conversely, suppose that $[g,D_G (g)] \nleq K$ for all $g \in G \setminus Z$.  Hence, we have $[gK,D_{G/K} (gK)] \ne 1$ for all $gK \in G/K \setminus Z/K$.  This implies that $[gK,G/K] \ne 1$ for all cosets $gK \in G/K \setminus Z/K$, and so ${Z} (G/K) \le Z/K$.  Since $Z/K \le {Z}(G/K)$ obviously holds, we have ${Z} (G/K) = Z/K$.  Notice that $\lambda$ is a faithful character of $Z/K$, so we may apply Lemma \ref{fr 1} to see that $\lambda$ is fully ramified with respect to $\displaystyle \frac {G/K}{Z/K} \cong \frac GZ$. 
\end{proof}

Let $G$ be a group, fix a character $\chi \in \irr G$, and write $\chi_{Z(G)} = \chi(1)\lambda$ for some character $\lambda \in \irr {Z(G)}$.  Note that $\ker(\lambda) = \ker(\chi)\cap Z(G)$.  Consider an element $g\in G$.  Since $[g,D_G (g)] \le Z(G)$, we have $[g,D_G(g)] \nleq \ker(\lambda)$ if and only if $[g,D_G (g)] \nleq \ker(\chi)$.   Furthermore, $[g,D_G (g)] \nleq \ker (\chi)$ if and only if there exists an element $x \in G$ so that $[g,x] \in Z(G) \setminus \ker (\chi)$.  Hence, Lemma~\ref{fullyram} can be equivalently stated as follows.

\begin{lem}\label{fullyram2}
Let $G$ be a group. A character $\chi \in \irr G$ is fully ramified over $Z(G)$ if and only if for every element $g \in G \setminus {Z}(G)$, there exists an element $x \in G$ for which $[g,x] \in Z(G) \setminus \ker (\chi)$.	
\end{lem}

This yields the desired characterization of central type characters.

\begin{thm}\label{fullyramthm}
The character $\chi \in \irr G$ has central type if and only if for every element $g \in G \setminus {Z}(\chi)$, there exists an element $x \in G$ for which $[g,x] \in Z(\chi) \setminus \ker (\chi)$.
\end{thm}

\begin{proof}
Note that $\chi$ is a faithful irreducible character of $G/\ker (\chi)$ and ${Z} (G/\ker (\chi)) = {Z} (\chi)/\ker (\chi)$. Thus we see from Lemma \ref{fullyram2} that $\chi$, regarded as a characer of $G/\ker(\chi)$, has central type if and only if for every element $g\in G\setminus Z(\chi)$, there exists an element $x\in G$ for which $1\ne [g,x]\ker(\chi)\in Z(G/\ker(\chi))$. It is easy to see that this is equivalent to the statement that was to be proved. 
\end{proof} 

\begin{rem}
	Observe that Theorem~\ref{fullyramthm} implies the well-known result that $\chi$ has central type if $G/Z(\chi)$ is abelian (see \cite[Theorem 2.31]{MI76}, for example).
\end{rem}

Before proceeding, we discuss monolithic groups and characters. 
We need one more result to prove Theorem~\ref{gvz thm}.  This result is proved in our paper \cite{ourpre3}. 

\begin{thm}\label{mono nilp}
The group $G$ is nilpotent if and only if $Z(\chi) > \ker(\chi)$ for each nonprincipal, monolithic character $\chi\in\irr{G}$.
\end{thm}

	

We now prove Theorem \ref{gvz thm}.

\begin{proof}[Proof of Theorem~\ref{gvz thm}]
First note the the statement (1) implies (2) follows immediately from Theorem~\ref{fullyramthm}.
	
Next we show that (2) implies (3). Let $\chi\in\irr{G}$ be monolithic. By Theorem~\ref{fullyramthm}, $\chi$ has central type. In particular $\chi(1)^2=\norm{G:Z(\chi)}$, from which we deduce that $Z(\chi)>\ker(\chi)$ if $\chi$ is nonprincipal. Thus $G$ is nilpotent by Theorem~\ref{mono nilp}. Now, let $N$ be a normal subgroup of $G$ for which $G/N$ is monolithic. Then $G/N$ has a faithful irreducible character, and thus $N=\ker(\chi)$ for some character $\chi\in\irr{G}$. Let $g\in G$ such that $[g,G]\nleq N$. Then $gN\notin Z(G/N)=Z(\chi)/N$ and so $g\notin Z(\chi)$. By (1), there exists $x\in G$ such that $[g,x]\in Z(\chi)\setminus N$. Since $Z(\chi)/N=Z(G/N)$, we see that $[[g,x],G]\le N$. 
	
To complete the proof, we show that (3) implies (1). Fix a prime $p$ that divides $\norm{G}$, a Sylow subgroup $P \in \mathrm{Syl}_p (G)$, and a character $\psi \in \irr P$.  Consider the character $\xi = \psi \times \mathbbm{1}_H \in \irr G$, where $H$ is a normal $p$-complement of $G$. Then $G/\ker(\xi) \cong P/\ker(\psi)$ is monolithic, by \cite[Theorem 2.32]{MI76}.  So $\xi$ is fully ramified over $Z (\xi) = Z(\psi) \times H$ by Theorem~\ref{fullyramthm}, and this implies that $\psi$ is fully ramified over $Z (\psi)$.  Now, consider a character $\chi \in \irr G$. To show that  $G$ is a GVZ-group, it suffices to show that $\chi$ is fully ramified over $Z (\chi)$.  Suppose that $G = P_1 \times \dotsb \times P_r$ is a factorization of $G$ into a direct product of its Sylow subgroups. Then there exist characters $\nu_i \in \irr {P_i}$ so that $\chi =\nu_1 \times \dotsb \times \nu_r$.  Observe that $Z(\chi) = Z(\nu_1) \times \dotsb \times Z(\nu_r)$.  We have already shown that each $\nu_i$ is fully ramified over $Z(\nu_i)$ and so it follows that $\chi$ is fully ramified over $Z(\chi)$, as desired.  This proves (1).
\end{proof}


\begin{thebibliography}{99}
	
	

\bibitem{ourpre3} S. T. Burkett and M. L. Lewis, Characters with nontrivial center modulo their kernel, preprint.
	
\bibitem{ourpre} S. T. Burkett and M. L. Lewis, GVZ-groups, flat groups, and CM-groups, preprint.

\bibitem{ourpre2} S. T. Burkett and M. L. Lewis, Partial GVZ-groups, preprint.
	
	
	
	
\bibitem{DJ} F. R. DeMeyer and G. J. Janusz, Finite groups with an irreducible character of large degree, Math. Z. 108 (1969), 145-153.
	
\bibitem{espuelas}  A. Espuelas, On certain groups of central type, Proc. Amer. Math. Soc. 97 (1986), 16-18.
	
\bibitem{gagola} S. M. Gagola, Jr., Characters fully ramified over a normal subgroup, Pacific J. Math. 55 (1974), 107-126.
	
\bibitem{IH} R. B. Howlett and I. M. Isaacs, On groups of central type, Math. Z. 179 (1982), 552-569. 
	
\bibitem{MI76} I. M. Isaacs, Character theory of finite groups, Dover Publications, Inc., New York, 1994.
	
\bibitem{MI18} I. M. Isaacs, Characters of solvable groups, American Mathematical Society, Providence, RI, 2018.
	
	
	
	
	\bibitem{ML19gvz} M.~L. Lewis, Groups where the centers of the irreducible characters form a chain, { Monatsh. Math.} 192 (2020), 371--399.
	
	
	
	
	\bibitem{OnoType} M. Murai, \newblock Characterizations of {$p$}-nilpotent groups.  Osaka J. Math. 31 (1994), 1--8.
	
	
	\bibitem{AN12gvz} A. Nenciu, Isomorphic character tables of nested GVZ-groups, J. Algebra Appl. 11 (2012), 1250033, 12 pp.
	
	\bibitem{AN16gvz} A. Nenciu, Nested GVZ-groups, J. Group Theory 19 (2016),  693-704.
	
	\bibitem{flatness} H. Tandra, and W. Moran, Flatness conditions on finite {$p$}-groups, {Comm. Algebra}, {32} (2004), 2215--2224.
	
	
	
	
	
	
\end{thebibliography}
\end{document}

